\documentclass[11pt]{elsarticle}%{amsart}
\usepackage{mathpazo}
\usepackage{times}

\usepackage{amsthm}
\usepackage{amsmath}
\usepackage{amsfonts}
\usepackage{amssymb}
\usepackage{mathrsfs} % Define \mathscr, a script font

\usepackage{upref}

\usepackage{graphicx}

\usepackage{fullpage}

\usepackage[utf8]{inputenc}
\usepackage[english]{babel}
\usepackage{amsthm}
\usepackage{color}
\usepackage{verbatim}

\newtheorem{theorem}{Theorem} %[section]
\newtheorem{lemma}[theorem]{Lemma}
\newtheorem{claim}[theorem]{Claim}
\newtheorem{definition}[theorem]{Definition}

\newcommand{\Tref}[1]{The\-o\-rem\,\ref{#1}}

\newcommand{\Cref}[1]{Co\-ro\-lla\-ry\,\ref{#1}}

\newcommand{\be}[1]{\begin{equation}\label{#1}}
\newcommand{\ee}{\end{equation}}
\newcommand{\eq}[1]{(\ref{#1})}

\renewcommand{\le}{\leqslant}
\renewcommand{\leq}{\leqslant}
\renewcommand{\ge}{\geqslant}
\renewcommand{\geq}{\geqslant}

\def\Z{\mathbb{Z}}
\def\B{\mathcal{B}}
\def\Q{\mathbb{Q}}
\def\L{\mathcal{L}}
\def\E{\mathbb{E}}
\def\R{\mathbb{R}}
\def\C{\mathbb{C}}

\newcommand{\bfit}{\bfseries\itshape}

\newcommand{\LS}{\textup{\textsf{LS}}}

\def\sym{\text{Sym}}
\def\vol{\text{vol}}
\def\eps{\varepsilon}
\def\x{\mathbf{x}}
\def\r{\mathbf{r}}

\newcommand{\ip}[1]{\langle #1 \rangle}
\newcommand{\ii}[1]%
    {{\left\vert\kern-0.25ex\left\vert #1 \right\vert\kern-0.25ex\right\vert}}
\newcommand{\iii}[1]%
     {{\left\vert\kern-0.25ex\left\vert\kern-0.25ex\left\vert %
            #1 \right\vert\kern-0.25ex\right\vert\kern-0.25ex\right\vert}}

\newcommand{\cD}{{\cal D}}

\makeatletter

\gdef\@punct{.\ \ }  % Punctuation after run-in section heading
\def\@sect#1#2#3#4#5#6[#7]#8{%
  \ifnum #2>\c@secnumdepth
     \def\@svsec{}
  \else
     \refstepcounter{#1}\edef\@svsec{%
     \ifnum #2>0{{\csname the#1\endcsname}}.\fi%
    \hskip .5em}
  \fi
  \@tempskipa #5\relax
  \ifdim \@tempskipa>\z@
     \begingroup #6\relax
       \@hangfrom{\hskip #3\relax\@svsec}{\interlinepenalty \@M #8\par}
     \endgroup
     \csname #1mark\endcsname{#7}
     \addcontentsline{toc}{#1}{\ifnum #2>\c@secnumdepth\else
          \protect\numberline{\csname the#1\endcsname}\fi#7}
  \else
     \def\@svsechd{#6\hskip #3\@svsec #8\@punct\csname #1mark\endcsname{#7}
     \addcontentsline{toc}{#1}{\ifnum #2>\c@secnumdepth \else
          \protect\numberline{\csname the#1\endcsname}\fi#7}}
  \fi
  \@xsect{#5}}

\def\@ssect#1#2#3#4#5{\@tempskipa #3\relax
  \ifdim \@tempskipa>\z@
    \begingroup #4\@hangfrom{\hskip #1}{\interlinepenalty \@M #5\par}\endgroup
  \else \def\@svsechd{#4\hskip #1\relax #5\@punct}\fi
  \@xsect{#3}}

\makeatother

\begin{document}

\begin{frontmatter}
\title{Probabilistic Existence of Large Sets of Designs}

\author{Shachar Lovett}
\ead{slovett@ucsd.edu}

\author{Sankeerth Rao}
\ead{sankeerth1729@gmail.com}

\author{Alexander Vardy}
\ead{avardy@ucsd.edu}

\address{University of California San Diego, 9500 Gilman Drive, La Jolla, CA 92093}

\begin{abstract}
\noindent\looseness=-1
A new probabilistic technique for establishing the existence
of certain regular combinatorial structures has been %R recently
introduced by Kuperberg, Lovett, and Peled (STOC 2012).
Using this technique, it can be shown that under certain
conditions, a randomly chosen structure has the required
properties of a $t$-$(n,k,\lambda)$ combinatorial design
with tiny, yet positive, probability.

The proof method of KLP is adapted to show the existence of large sets of designs and similar combinatorial structures as follows.
We modify the random choice and the analysis
to show that, under the same conditions, not only does
a $t$-$(n,k,\lambda)$ design exist but, in fact, with
positive probability there exists a \emph{large set} of
such designs --- that is, a partition of the set of $k$-subsets
of $[n]$ into %R $t$-designs
$t$-$(n,k,\lambda)$ designs.
Specifically, using the probabilistic approach derived % introduced
herein, we prove that for all sufficiently large $n$,
large sets of $t$-$(n,k,\lambda)$ designs exist whenever $k > 12t$
and the necessary divisibility conditions are satisfied.
This resolves the existence conjecture for large sets of
designs for all $k > 12t$.
\end{abstract}


\end{frontmatter}

\setcounter{page}{1}
%==============================================================================%
%                                                                              %
%   1. INTRODUCTION                                                            %
%                                                                              %
%==============================================================================%
\section{Introduction}
\label{sec:intro}

Let $[n]=\{1,2,\ldots,n\}$. A $k$-set is a subset of $[n]$ of size $k$.
A \emph{$t$-$(n,k,\lambda)$ combinatorial design}
is a collection $\cD$ of distinct $k$-sets of $[n]$,
called \emph{blocks}, such that every $t$-set
of $[n]$ is contained in exactly $\lambda$ blocks.
A~\emph{large set of %R $t$-$(n,k,\lambda)$
designs} of size $l$, denoted $\LS(l;\, t,k,n)$,
is a set of\, $l$ disjoint $t$-$(n,k,\lambda)$ designs
$\cD_1,\cD_2,\ldots,\cD_l$ such~that
$\cD_1 \cup \cD_2 \cup \cdots \cup \cD_l$
is the set of all $k$-sets of $[n]$.
That is, $\LS(l;\,t,k,n)$ is a partition of
the set of $k$-sets of $[n]$ into $t$-$(n,k,\lambda)$
designs, where necessarily $\lambda = \binom{n-t}{k-t}/l$.

The {\bfit existence problem\/} for large sets of designs can be
phrased as follows: for which values of\, $l,t,k,n$ do
$\LS(l;\,t,k,n)$ large sets exist?
The {\bfit existence conjecture\/} for large sets,
formulated for example in \cite[Conjecture\,1.4]{Teirlinck92}, asserts
that for every fixed $l,t,k$ with $k \ge t+1$, a large
set $\LS(l;\,t,k,n)$ exists for~all~sufficiently large $n$
that satisfy the obvious divisibility constraints (see
Section\,\ref{divisibility-constraints}). However,
according to \cite[p.\,564]{Teirlinck92}~as well as %R other
more recent surveys, ``not many results about $\LS(l;\,t,k,n)$
with $k > t+1$ are known.''
One of our main results herein is a proof of the foregoing
existence conjecture for all $k > 12t$.

\vspace{0.75ex}
%------------------------------------------------------------------------------%
%                                                                              %
%    1.1. Large sets of designs                                                %
%                                                                              %
%------------------------------------------------------------------------------%
\subsection{Large sets of designs}\vspace{-0.50ex}
%R \subsection{Large sets of designs: History and applications}\vspace{-0.50ex}
\label{large-sets}

\noindent
Combinatorial design theory can be traced back to the work of Euler,
who introduced the famous ``36 officers problem'' in 1782. Euler's
ideas were further developed in the mid-19th century by Cayley,
Kirkman, and Steiner. In particular, the existence problem for
large sets of designs
was first considered in 1850~by~Cayley~\cite{Cayley1850},
who found two disjoint $2$-$(7,3,1)$ designs
and showed that no more exist.
The first nontrivial large set, namely $\LS(7;\,2,3,9)$,
was constructed by Kirkman~\cite{Kirkman1850} in the same year.
Following these results,~the existence problem for large sets
of type $\LS(n{-}2;\,2,3,n)$ --- that is, large sets of Steiner
triple systems --- attracted %R gained
considerable research attention. Nevertheless, this problem
remained open until the 1980s, when it was settled by Lu~\cite{Lu83,Lu84}
and Teirlinck~\cite{Teirlinck91}. Specifically, it is %R was
shown in~\cite{Lu83,Lu84,Teirlinck91} that $\LS(n{-}2;\,2,3,n)$ exist
for all $n \ge 9$ with $n \equiv 1,3\! \pmod 6$.
In 1987 came the celebrated work %R result
of Teirlinck~\cite{Teirlinck87}, who proved
that nontrivial $t$-$(n,k,\lambda)$ designs exist for all values of $t$.
In fact, % he did so
Teirlinck's proof of this theorem in~\cite{Teirlinck87}~proceeds
by constructing for all $t \ge 1$, a large set %R the large sets
$\LS(l;\,t,t\,{+}\,1,n)$, where $l = (n-t)/(t\,{+}\,1)!^{(2t+1)}$.
%R Teirlinck's
His~results in~\cite{Teirlinck87,Teirlinck89} further
imply that for all fixed $t,k$ with $k \ge t{+}1$, nontrivial
large sets $\LS(l;\,t,k,n)$ exist for infinitely many values of $n$.
However, as mentioned earlier, it is unknown whether such large sets
exist for \emph{all sufficiently large values of $n$} that satisfy
the necessary divisibility constraints. For much more on the history
of the problem and the current state of knowledge, see the surveys
%R by Teirlinck~\cite{Teirlinck92} and by Khosrovshahi and Tayfeh-Rezaie.
\cite{Teirlinck92,KL06,KTR06} and references therein.

There are numerous applications of large sets of designs in
discrete mathematics and computer science. For example,
large sets of Steiner systems were used to construct perfect
secret-sharing schemes by Stinson and Vanstone~\cite{SV88}
and follow up works~\cite{SS89,Etzion96}. An application of
general large sets of designs to threshold secret-sha\-ring
schemes was proposed by Chee~\cite{Chee89}.
As another example, Chee and Ling~\cite{CL07}
showed how large sets %of $t$-designs
can be used to construct infinite families
of optimal constant weight codes.
As yet another example, large sets of $1$-designs
(also known as one-factorizations) have been used
extensively in various kinds of scheduling problems
--- see~\cite[pp.\,51--53]{MR85} and references therein.

\vspace{1.25ex}
%------------------------------------------------------------------------------%
%                                                                              %
%    1.2. Divisibility constraints and existence theorem                       %
%                                                                              %
%------------------------------------------------------------------------------%
\subsection{Divisibility constraints and our existence theorem}\vspace{-0.25ex}
\label{divisibility-constraints}

\noindent
Consider a $t$-$(n,k,\lambda)$ design with $N$ blocks. It is very easy
to see that every such design must satisfy certain natural divisibility
constraints. For instance, every $k$-set of $[n]$ contains exactly
$\binom{k}{t}$ many $t$-sets, and since every $t$-set is covered
exactly $\lambda$ times by the $N$ blocks, %R of the design,
we % must
have \smash{$N \binom{k}{t} = \lambda \binom{n}{t}$}.
In~particular, this implies that
$\binom{k}{t}$ should divide $\lambda \binom{n}{t}$.
Now let us fix a positive integer $s \le t\,{-}\,1$ and restrict
\pagebreak[3.99]
our attention only to those $N'$ blocks that contain a specific
$s$-set of $[n]$. Since the fixed $s$-set can be extended to a $t$-set
in $\binom{n-s}{t-s}$ ways and each of these $t$-sets is covered $\lambda$
times by the $N'$ blocks, a similar argument yields
\smash{$N' \binom{k-s}{t-s} = \lambda \binom{n-s}{t-s}$}.
Thus $\binom{k-s}{t-s}$ should divide $\lambda\binom{n-s}{t-s}$.
Altogether, this simple counting argument produces
%R the following
$t$ divisibility constraints:~
\be{eq:div_designs}
\binom{k-s}{t-s}
%~~\text{divides}~~
~\Bigg|~~
\lambda \binom{n-s}{t-s}
\quad\quad
\text{for all~ $s=0,1\ldots,t-1$}.
\ee

The above leads to the following natural question.
Are these $t$ divisibility conditions %R constraints
also \emph{sufficient} for~the existence
of $t$-$(n,k,\lambda)$ designs, at least
when $n$ is large enough? This is one of the central
questions~in~combinatorial design theory.
In a remarkable achievement, Keevash~\cite{Keevash14}
was able to answer this question positively,
thereby settling the \emph{existence conjecture for
combinatorial designs}. Specifically, Keevash
proved %~in \cite{Keevash14}
that for any $k > t \geq 1$ and $\lambda \ge 1$,
there is a sufficiently large $n_0=n_0(t,k,\lambda)$ such
that the following holds:
for all $n \ge n_0$ such that $n,t,k,\lambda$
satisfy the divisibility conditions in \eq{eq:div_designs},
there exists a $t$-$(n,k,\lambda)$ design.

Let us now consider the divisibility conditions for
large sets. % of designs.
A large set $\LS(l;\,t,k,n)$ is a partition of all $k$-sets of $[n]$
into $t$-$(n,k,\lambda)$ designs. Clearly, each of these designs
consists of
$
N = \binom{n}{k}/l = \lambda \binom{n}{t} / \binom{k}{t}
$
blocks. This can be used to specify $\lambda$ in terms of
$n,t,k,l$ as follows:
\be{eq:lambda}
\lambda
\ = \
\frac{\displaystyle\binom{n}{k} \binom{k}{t}}{\displaystyle l \binom{n}{t}}
\ = \
\frac{1}{l} {{n-t}\choose{k-t}}
\ee
With this, the divisibility constraints \eq{eq:div_designs} for the
$l$ component designs of a large set $\LS(l;\,t,k,n)$
can~be~re-written in terms of $n,t,k,l$.
Altogether, we conclude that the parameters
of a large set $\LS(l;\,t,k,n)$ must satisfy
the following $t+1$ divisibility constraints:
\be{eq:div2}
l\binom{k-s}{t-s}
%~~\text{divides}~~
~\Bigg|~~
\binom{n-t}{k-t} \binom{n-s}{t-s}
\quad\quad
\text{for all~ $s=0,1\ldots,t$}.
\ee
\looseness=-1
Note that the constraint for $s=t$ simply refers to the condition that $l$
must divide $\binom{n-t}{k-t}$, which is clearly necessary
in view of \eqref{eq:lambda}.
Once again, this leads to the following natural question.
Are these $t+1$ divisibility conditions %R constraints
also \emph{sufficient} for the existence of
$\LS(l;\,t,k,n)$ large sets, % of designs,
at least when $n$ is large enough?

One of our main results in this paper is a positive answer
to this question for all $k > 12t$, which settles the
existence conjecture for large sets for such %R these
values of $k$. We formulate this result as the following theorem.
\begin{theorem}
\label{thm:largesets}
For any $t \ge 1, k > 12t$ and $l \ge 1$,
there is an %R a sufficiently large
$n_0=n_0(t,k,l)$ such
that the following holds:~for all $n \ge n_0$ such that $n,t,k,l$
satisfy the divisibility conditions in\/ \textup{\eq{eq:div2}},
there exists an\/ $\LS(l;\,t,k,n)$ large set.~
\end{theorem}

\looseness=-1
In fact, \Tref{thm:largesets} %R easily
follows as a special case of
a more general statement --- namely, \Tref{thm:main} of
Section\,\ref{main-theorem}. \Tref{thm:main} itself follows
by adapting the probabilistic argument of Kuperberg, Lovett, and Peled~\cite{KLP12} to show the existence of large sets of designs and similar
combinatorial structures.
We begin by describing the general framework for this probabilistic argument
below. %R in what follows.

\vspace{1.25ex}
%------------------------------------------------------------------------------%
%                                                                              %
%    1.2. Divisibility constraints and existence theorem                       %
%                                                                              %
%------------------------------------------------------------------------------%
\subsection{General framework}\vspace{-0.25ex}
\label{general-framework}

Throughout this work, we will use the notation of the
Kuperberg, Lovett, and Peled paper~\cite{KLP12}, which we shorthand as KLP.
Let $A,B$ be finite sets and let $\phi:B \to \Z^A$ be a vector valued function. One can think of $\phi$ as described by a $|B|\times |A|$ matrix where the rows correspond to the evaluation of the function $\phi$ on the elements in $B$. In this setting \cite{KLP12} gives sufficient conditions for the existence of a small set $T\subset B$ such that
\begin{equation}\label{eq:avg_orig}
\frac{1}{|T|}\sum_{t\in T}\phi(t) = \frac{1}{|B|}\sum_{b \in B} \phi(b).
\end{equation}
In the context of designs we can think of $B$ as all the $k$-sets of $[n]$ and $A$ as all the $t$-sets of $[n]$. $\phi$ denotes the inclusion function, that is $\phi(b)_a = 1_{a \subset b}$
where $b$ is a $k$-set of $[n]$ and $a$ is a $t$-set of $[n]$.
Equation~\eqref{eq:avg_orig} is then equivalent to $T$ being
a $t$-$(n,k,\lambda)$ design for an appropriate $\lambda$.

Next, we present the conditions under which KLP showed that there is a solution for~\eqref{eq:avg_orig}. We start with a few useful notations.
For $a \in A$ we denote by $\phi_a \in \Z^B$ the $a$-column of the matrix described by $\phi$, namely $(\phi_a)_b = \phi(b)_a$.
Let $V \subset \Q^B$ be the vector space over $\Q$ spanned by the columns of this matrix $\{\phi_a: a \in A\}$.
Observe that \eqref{eq:avg_orig} depends only on $V$ and not on $\{\phi_a: a \in A\}$, which is a specific choice of basis for $V$.
We identify $f \in V$ with a function $f:B \to \Q$. Thus, we may reformulate \eqref{eq:avg_orig} as
\begin{equation}\label{eq:avg}
\frac{1}{|T|}\sum_{t\in T}f(t) = \frac{1}{|B|}\sum_{b \in B} f(b) \quad \forall f \in V.
\end{equation}
In particular, we may assume without loss of generality that $\dim(V) = |A|$.

The conditions and results outlined below will depend only on the subspace $V$. However, it will be easier to present some of them with a specific choice of basis.
We may assume this to be an integer basis (A basis of the subspace $V$ made up of vectors with only integer coordinates). Thus, we assume throughout that $\phi:B \to \Z^A$ is a map whose coordinate projections $\phi_a:B \to \Z$ are a basis for $V$.

\subsubsection{Divisibility conditions}

For $T$ to be a valid set for \eqref{eq:avg} with $|T| = N$, we must have
$$
\sum_{t\in T}f(t) = \frac{N}{|B|}\sum_{b\in B}f(b) \quad \forall f \in V.
$$
In particular there must exist $\gamma \in \Z^B$ such that
\begin{equation}\label{eq:div}
\sum_{b\in B}\gamma_b f(b) = \frac{N}{|B|}\sum_{b\in B} f(b) \quad \forall f \in V.
\end{equation}
The set of integers $N$ satisfying \eqref{eq:div} for some $\gamma \in \Z^B$ consists of all integer multiples of some minimal positive integer $c_1$. This is because if $N_1$ and $N_2$ are solutions  then so is $N_1-N_2$. Thus it  follows that $|T|$ must be an integer multiple of $c_1$. This is the divisibility condition and $c_1$ is the divisibility parameter of $V$.

We can rephrase~\eqref{eq:div} as $\frac{N}{|B|}\sum_{b\in B} \phi(b)$ belongs to the lattice spanned by $\{\phi(b): b \in B\}$.

\begin{definition}[Lattice spanned by $\phi$]
We define $\L(\phi)$ to be the lattice spanned by $\{\phi(b): b\in B\}$.
$$
\L(\phi) = \Big \{ \sum_{b \in B}n_b\cdot\phi(b): n_b \in \Z \Big\} \subset \Z^A.
$$
\end{definition}

Note that since we assume that $\dim(V) = |A|$ we have that $\L(\phi)$ is a full rank lattice.

\begin{definition}[Divisibility parameter $c_1$]
The divisibility parameter of $V$ is the minimal integer $c_1 \ge 1$ that satisfies $\frac{c_1}{|B|} \sum_{b \in B} \phi(b) \in \L(\phi)$. Note that
it does not depend on the choice of basis for $V$ which defines $\phi$.
\end{definition}

\subsubsection{Boundedness conditions}

The second condition is about boundedness conditions for integer vectors which span $V$ and its orthogonal dual. We start with some general definitions. Let $1 \le p < \infty$.
The $\ell_p$ norm of a vector $\gamma \in \Z^B$ is $\|\gamma\|_p = (\sum_{b \in B} |\gamma_b|^p)^{1/p}$. Below we restrict our attention to
$\|\gamma\|_1 = \sum_{b \in B} |\gamma_b|$ and $\|\gamma\|_{\infty} = \max_{b \in B} |\gamma_b|$.

\begin{definition}[Bounded integer basis]
Let $W\subset \Q^B$ be a vector space. For $1\leq p \leq \infty$, we say that W has a c-bounded integer basis in $\ell_p$ if $W$ is spanned by integer vectors whose $\ell_p$ norm is at most $c$.  That is, if
$$Span(\{\gamma \in W \cap \Z^B: \|\gamma\|_p\leq c\}) = W.$$
\end{definition}

Recall that $V \subset \Q^B$ is the vector space over $\Q$ spanned by $\{\phi_a: a \in A\}$.
We denote by $V^{\perp}$ the orthogonal complement of $V$ in $\Q^B$, that is,
$$
V^{\perp} :=  \{g \in \Q^B: \sum_{b\in B}f(b)g(b) = 0 \quad \forall f \in V \}.
$$

\begin{definition}[Boundedness parameters $c_2,c_3$]
We impose two boundedness conditions:
\begin{itemize}
\item Let $c_2 \ge 1$ be such that $V$ has a $c_2$-bounded integer basis in $\ell_{\infty}$.
\item Let $c_3 \ge 1$ be such that $V^{\perp}$ has a $c_3$-bounded integer basis in $\ell_1$.
\end{itemize}
\end{definition}

\subsubsection{Symmetry conditions}

Next we require some symmetry conditions from the space $V$. Given a permutation $\pi \in S_B$ and a vector $f \in \Q^B$, we denote by $\pi(f) \in \Q^B$
the vector obtained by permuting the coordinates of $f$, namely $\pi(f)_b = f_{\pi(b)}$.

\begin{definition}[Symmetry group of $V$]
The symmetry group of $V$, denoted $\sym(V)$, is the set of all permutations $\pi \in S_B$ which satisfy that $\pi(f) \in V$ for all $f \in V$.
\end{definition}

It is easy to verify that $\sym(V)$ is a subgroup of $S_B$, the symmetric group of permutations on $B$.
Note that the condition $\pi \in \sym(V)$ can be equivalently cast as the existence of an invertible linear map $\tau:\Q^A\to\Q^A$ such that
$$
\phi(\pi(b))= \tau(\phi(b)) \qquad \forall \; b \in B.
$$

\begin{definition}[Transitive symmetry group]
The symmetry group of $V$ is said to be transitive if it acts transitively on $B$. That is, for every $b_1,b_2 \in B$ there is $\pi \in \sym(V)$ such that $\pi(b_1)=b_2$.
\end{definition}

\subsubsection{Constant functions condition}
The last condition is very simple: we require that the constant functions belong to $V$.

\subsubsection{Main theorem of KLP}

We are now at a position to state the main theorem of KLP~\cite{KLP12}.

\begin{theorem}[KLP Theorem]
\label{thm:klp}
Let $B$ be a finite set and let $V \subset \Q^B$ be the subspace of functions. Assume that the following holds for some integers $c_1,c_2,c_3 \geq 1$:
\begin{itemize}
\item Divisibility: $c_1$ is the divisibility parameter of $V$.
\item Boundedness of $V$: $V$ has a $c_2$-bounded integer basis in $\ell_{\infty}$.
\item Boundedness of $V^{\perp}$: $V^{\perp}$ has a $c_3$-bounded integer basis in $\ell_1$.
\item Symmetry: $V$ has a transitive symmetry group.
\item Constant functions: The constant functions belong to V.
\end{itemize}
Let $N$ be an integer multiple of $c_1$ satisfying
$$
min(N,|B|-N)\geq C\cdot c_2c_3^2 dim(V)^6\log(2c_3dim(V))^6,
$$
where $C>0$ is an absolute constant. Then there exists a subset $T \subset B$ of size $|T| = N$ satisfying
$$
\frac{1}{|T|}\sum_{t \in T}\phi(t) = \frac{1}{|B|}\sum_{b\in B}\phi(b).
$$
\end{theorem}

\subsection{Our main theorem}
\label{main-theorem}

Our main result is an extension of the KLP theorem (Theorem~\ref{thm:klp}) to large sets. It will have many of the same conditions, except that we need to
update the divisibility condition to require the size of each design to be $N=|B|/\ell$. Thus the new divisibility condition is
$$
\frac{1}{l} \sum_{b \in B} \phi(b) \in \L(\phi).
$$
Note that as before, this condition depends only on $V$; it does not depend on the choice of basis for $V$ which defines $\phi$.

\begin{theorem}[Main theorem]
\label{thm:main}
Let $B$ be a finite set and let $V \subset \Q^B$ be the subspace of functions. Let also $l \ge 1$ be an integer.
Assume that the following holds for some integers $c_2,c_3 \geq 1$:
\begin{itemize}
\item Divisibility: $\frac{1}{l} \sum_{b \in B} \phi(b) \in \L(\phi)$.
\item Boundedness of $V$: $V$ has a $c_2$-bounded integer basis in $\ell_{\infty}$.
\item Boundedness of $V^{\perp}$: $V^{\perp}$ has a $c_3$-bounded integer basis in $\ell_1$.
\item Symmetry: The symmetry group of $V$ is transitive.
\item Constant functions: The constant functions belong to V.
\end{itemize}
Assume furthermore that
$$
|B| \ge C \dim(V)^6 l^7 c_3^3 \log^3(\dim(V) c_2 c_3 l),
$$
for some absolute constant $C>0$.
Then there exists a partition of $B$ to $T_1,\ldots,T_l$, each of size $|T_i|=|B|/l$ such that
$$
\sum_{t \in T_i}\phi(t) = \frac{1}{l}\sum_{b\in B}\phi(b) \quad \text{for all} \quad i=1,\ldots,l.
$$
\end{theorem}

Theorem~\ref{thm:largesets} follows as a special case of Theorem~\ref{thm:main}.

\begin{proof}[Proof of Theorem~\ref{thm:largesets}]
To recall, in this setting we have $B$ the set of all $k$-sets of $[n]$,
$A$ the set of all $t$-sets of $[n]$, $\phi:B \to \{0,1\}^A$ given by inclusion $\phi(b)_a = 1_{a \subset b}$ for $a \in A, b \in B$ and $V$ the subspace
spanned by $\{\phi_a: a \in A\}$.

KLP~\cite{KLP12} showed (see Section 3.3 in the arxiv version)
that in this setting, the subspace $V$ has a transitive symmetry group, it contains the constant functions, and it has boundedness parameters $c_2=1, c_3 \le (4e n/t)^{2t}$. Furthermore,
the condition that the vector $\bar{\lambda} = (\lambda,\ldots,\lambda) \in \L(\phi)$ is equivalent to the set of conditions
$$
{k-s \choose t-s} \bigg| \lambda {n-s \choose t-s} \quad \text{for all} \quad s=0,\ldots,t.
$$
(see Theorem 3.7 in~\cite{KLP12}). In particular in our case $\lambda = {n-t \choose k-t}/l$ and hence the divisibility conditions in Theorem~\ref{thm:main}
are equivalent to the necessary divisibility conditions given in \eqref{eq:div2}. To obtain the lower bound on $|B|$, fix $k,t,l$ and let $n$ be large enough.
Then $|B| = \Theta(n^k)$, $\dim(V) = \Theta(n^t)$ and $c_3 = \Theta(n^{2t})$. Then if $k > 12t$ and $n$ is large enough the lower bound on $B$ holds.
\end{proof}
We conclude this section with a few remarks. First, Theorem~\ref{thm:largesets} is stated for fixed $\ell,t,k$ and large enough $n$. However, one can allow $\ell,t,k$ to grow as small polynomials in $n$ and Theorem~\ref{thm:largesets} still holds and follows from Theorem~\ref{thm:main}. Second, the proof of Theorem~\ref{thm:largesets} implies an analogous counting result similar to KLP, as it estimates the probability for the relevant event to hold. We do not include these calculations explicitly in the paper, but they can be readily derived from the proof of Theorem~\ref{thm:main}.
Last, in this paper we focus on the application of Theorem~\ref{thm:main} to large sets. However, Theorem~\ref{thm:main} readily applies for other applications mentioned in KLP, as the assumptions are the same, except for the added divisibility assumption for $\ell$.

\subsection{Proof overview}
The high level idea, similar to~\cite{KLP12}, is to analyze the natural random process and show that with positive (yet exponentially small) probability a desired event occurs.

Say that a subset $T \subset B$ is ``uniformly random" if
$$
\frac{1}{|T|} \sum_{b \in T}\phi(b)  = \frac{1}{|B|} \sum_{b \in B} \phi(b).
$$
Equivalently, the ``tests" defined by $V$ cannot distinguish the uniform distribution over $T$ from the uniform distribution over $B$.

Let $\tau:B \to [l]$ be a uniform partition of $B$ into $l$ sets. Let $T_i = \tau^{-1}(i)$ be the induced partition for $i=1,\ldots,l$.
We would like to analyze the event that each part is uniformly random. That is, we would like to show that
\begin{equation}
\label{eq:proofoverview1}
\Pr[T_1,\ldots,T_l \text{ are uniformly random}] > 0.
\end{equation}
Notice that under the same notations, the main result of~\cite{KLP12} can be formulated as
$$
\Pr[T_1 \text{ is uniformly random}] > 0.
$$

The random process can be modeled as a random walk on a lattice. For $i=1,\ldots,l$ let $X_i=\sum_{b \in T_i} \phi(b)$ be random variables taking values in $\Z^A$.
Let $\lambda = \E[X_1]=\ldots=\E[X_l] \in \Q^{|A|}$. Note that if $X_1=\ldots=X_{l-1}=\lambda$ then also $X_l=\lambda$. Let $X=(X_1,\ldots,X_{l-1}) \in \Z^{(l-1)|A|}$.
Thus we can reformulate~\eqref{eq:proofoverview1} as
\begin{equation}
\label{eq:proofoverview2}
\Pr[X = \E[X]] > 0.
\end{equation}
Recall that each random variable $X_i$ takes values in a full-dimensional sub-lattice of $\Z^A$ which we denoted $\L(\phi)$. One can show that $X$ takes values in the lattice
$\L(\phi)^{\otimes (l-1)}$, which is a full dimensional lattice in $\Q^{(l-1)|A|}$. In order to study the distribution of $X$,
we apply a local central limit theorem. The same approach was applied in~\cite{KLP12} in order to analyze the individual distribution of each $X_i$.
Here, we extend the method to analyze their joint distribution, namely the distribution of $X$. This is accomplished by a careful analysis of the Fourier coefficients of $X$, which in turn
relies on ``coding theoretic" properties of the space $V$. Given this coding theoretic properties, we show that
$\Pr[X = \E[X]]$ can be approximated by the density of a Gaussian process with the same first and second moment as $X$ at the point $\E[X]$. In particular, it is positive,
which establishes the existence result.

\subsection{Broader perspective}
The current work falls into the regime of ``rare events" in probabilistic analysis. It is very common that the probabilistic method, when applied to show that certain
combinatorial objects exist (such as expander graphs, error correcting codes, etc) shows that a random sample succeeds with high probability. The challenge then shifts to obtaining
explicit constructions of such objects, with efficient algorithmic procedures whenever relevant (e.g.\ efficient decoding algorithms for codes).

However, there are several scenarios where the ``vanilla" probabilistic method fails, and one is forced to develop much more fine tuned techniques to prove existence
of the desired combinatorial objects. The current work falls into the regime where the random process is the natural one, but the analysis is much more delicate. Other
examples of similar instances are the constructive proof of the Lov\'{a}sz local lemma (see e.g.~\cite{Mo09,MT10}), the works on interlacing families of polynomials (see e.g.~\cite{MSS13a,MSS13b}),
and the entire field of discrepancy theory (see e.g.\ the book~\cite{Ma09}). In each such instance, new methods were developed to prove existence of the relevant objects, that go beyond simple probabilistic analysis.

There are several families of problems in combinatorics, for which the only known constructions are explicit and of algebraic or combinatorial nature. For example, this is the case for all types
of local codes (such as locally testable, decodable, or correctable codes; PIR schemes; batch codes, and so on). It is also the case for
Zarenkiewicz-type Ramsey problems in graph theory, about maximal bipartite graphs without certain induced subgraphs. Another well known example is the existence of Hadamard matrices.
The lack of a probabilistic model for a solution may be seen as the reason why the existential
results known for these problems are very sparse and ad-hoc.

In the current work, we show that for the problem of existence of large sets, one can move beyond explicit ad-hoc constructions, such as the one of Teirlinck~\cite{Teirlinck91},
to a more rigorous understanding of when existence of large sets is possible. Of course, the next step in this line of research, after existence has been established,
is to find explicit constructions. We leave this question for future research. Another question is whether the existence result can be established for the full spectrum
of parameters, namely $k \ge t+1$ and any $\ell \ge 1$ (recall that our result requires that $k > 12t$). This seems to be possible by replacing the Gaussian estimate
by an estimate which uses higher moments of the distribution of the random variable being analyzed. We leave this also for future research.

\section{Preliminaries}
\label{sec:prelim}

Recall that $\phi:B \to \Z^A$ is a map, whose coordinate projections are $\phi_a: B \to \Z$. We defined $V$ to be the subspace of $\Q^B$ spanned by $\{\phi_a: a \in A\}$. We may
assume that that these form a basis for $V$, and hence $\dim(V)=|A|$.

Let $\tau : B \to [l]$ be a mapping that partitions $B$ into $l$ bins. Let $T_i := \{b \in B: \tau(b)=i\}$ for $i \in [\ell]$ be the induced partition of $B$.
In order to prove Theorem~\ref{thm:main} we are looking for a $\tau$ for which
\begin{equation}\label{eq:partition}
\sum_{b \in T_i}\phi(b) = \frac{1}{l}\sum_{b\in B}\phi(b) \quad \text{for all} \quad i=1,\ldots,l.
\end{equation}
Note that it suffices to require that \eqref{eq:partition} holds for $i=1,\ldots,l-1$, as then it automatically also holds for $i=l$. So from now on we only require that
\eqref{eq:partition} holds for the first $l-1$ bins. We will choose a uniformly random mapping $\tau$, and show that \eqref{eq:partition} holds with a positive probability. Note that $\tau$ has independent coordinates which are each uniform on $[\ell]$, and this makes $X$ a sum of independent random vectors.

We start with some definitions. Let $\Phi:B \times [l] \to \Z^{(l-1)|A|}$ be defined as follows. $\Phi(b,i)=(x_1,\ldots,x_{l-1})$, where $x_1,\ldots,x_{l-1} \in \Z^A$
are given by $x_j = \phi(b) \cdot 1_{i=j}$. Note that in particular $\Phi(b,l)=0$. Next, define a random variable $X \in \Z^{(l-1) |A|}$ as
$$
X := \sum_{b \in B}\Phi(b,\tau(b)).
$$
The mean of $X$ is
$$
\E[X] = \left(\frac{1}{l}\sum_{b\in B}\phi(b),..., \frac{1}{l}\sum_{b\in B}\phi(b) \right) \in \Q^{(l-1)|A|}.
$$
Thus, proving Theorem~\ref{thm:main} is equivalent to showing that
\begin{equation}\label{eq:prob_positive}
\Pr_{\tau}[X = \E[X]] > 0.
\end{equation}

We start by computing the covariance matrix of $X$.

\begin{claim}
\label{claim:covariance}
The covariance matrix of $X$ is the $(l-1) |A| \times (l-1) |A|$ positive definite matrix
$$
\Sigma[X] = R \otimes M
$$
where $R$ is the $|A| \times |A|$ positive definite matrix
$$
R_{a,a'} = \sum_{b \in B} \phi(b)_a \phi(b)_{a'}
$$
and $M$ is the $(l-1) \times (l-1)$ matrix
$$
M = \frac{1}{l^2}
\begin{bmatrix}
    (l-1)       & -1 & \dots & -1 \\
    -1       & (l-1) & \dots & -1 \\
    \vdots & \vdots & \ddots & \vdots\\
    -1       & -1 & \dots & (l-1)
\end{bmatrix}.
$$
\end{claim}

\begin{proof}
The random variables $\{\Phi(b,\tau(b)): b \in B\}$ are independent, thus their contribution to the covariance matrix of $X$ is additive. Fix $b \in B$.
We compute the contribution of $\Phi(b,\tau(b))$ to the $(a,i),(a',i')$ entry of $\Sigma[X]$, where $a,a' \in A$ and $i,i' \in [l-1]$. The second moment is
$$
\E_{\tau}[\Phi(b,\tau(b))_{a,i} \cdot \Phi(b,\tau(b))_{a',i'}] = \frac{1}{l} \phi(b)_{a} \phi(b)_{a'} \cdot 1_{i=i'}.
$$
The expectation product is
$$
\E_{\tau}[\Phi(b,\tau(b))_{a,i}] \cdot \E_{\tau}[\Phi(b,\tau(b))_{a',i'}] = \frac{1}{l^2} \phi(b)_{a} \phi(b)_{a'}.
$$
Thus
$$
\Sigma[X]_{(a,i), (a',i')} = \sum_{b \in B} \phi(b)_{a} \phi(b)_{a'} \left( \frac{1}{l} \cdot 1_{i=i'} - \frac{1}{l^2} \right) = R_{a,a'} \cdot M_{i,i'}
= (R \otimes M)_{(a,i),(a',i')}.
$$
\end{proof}

Similar to the proof in KLP we would be interested in the lattice in which $X$ resides. Recall that $\L(\phi)$ is the lattice in $\Z^{|A|}$ spanned by the image of $\phi$.
We similarly define $\L(\Phi)$.

\begin{definition}[Lattice spanned by $\Phi$]
We define $\L(\Phi)$ to be the lattice spanned by $\{\Phi(b,i): b\in B, i \in [l]\}$. Namely,
$$
\L(\Phi) := \left\{ \left(\sum_{b_1 \in B}n_{b_1}\cdot \phi(b_1),..,\sum_{b_j \in B}n_{b_j}\cdot \phi(b_j),..,\sum_{b_{l-1} \in B}n_{b_{l-1}}\cdot \phi(b_{l-1}) \right) : n_{b_j} \in \mathbb{Z}, j \in [l-1]  \right\}.
$$
\end{definition}
Note that since $\dim(V) = |A|$ then $\L(\phi)$ is a full rank lattice in $\Z^{|A|}$. Hence $\L(\Phi) = \L(\phi)^{\otimes (l-1)}$ is a full rank lattice in $\Z^{(l-1)|A|}$.

Similar to KLP we use Fourier analysis to study the distribution of $X$. The Fourier transform of $X$ is the function $\widehat{X}:\R^{(l-1)|A|}\to \C$ defined by
$$
\widehat{X}(\Theta) = \E_X[e^{2\pi i\ip{X,\Theta}}].
$$
Note that $\widehat{X}$ is periodic. Concretely, let $L(\Phi)$ denote the dual lattice to $\L(\Phi)$,
$$
L(\Phi) := \Big\{\Theta \in \R^{(l-1)|A|} : \ip{\Lambda,\Theta} \in \mathbb{Z} \quad \forall \Lambda \in \L(\Phi) \Big\}.
$$
Note that if $\Theta \in L(\Phi)$ then $\widehat{X}(\Theta+\Theta')=\widehat{X}(\Theta')$ for all $\Theta' \in \R^{(l-1)|A|}$, and $\widehat{X}(\Theta)=1$ iff $\Theta \in L(\Phi)$.
As $\L(\Phi)$ is a full rank lattice it follows that $L(\Phi)$ is also a full rank lattice and $\det(\L(\Phi))\det(L(\Phi)) = 1$.
Thus studying $\widehat{X}$ on any fundamental domain of $L(\Phi)$ would be sufficient to study the behavior of $\widehat{X}$ on $\R^{(l-1)|A|}$.
Similar to KLP we work with a natural fundamental domain defined by a norm related to the covariance matrix of $X$.

\begin{definition}[$R$-norm]
For $\Theta = (\theta_1,...,\theta_{l-1}) \in \R^{(l-1)A}$ we define the norm $\|\cdot\|_{R}$ as
$$
\|\Theta\|_R := \max_{j \in [l-1]} \left( \frac{1}{|B|} \theta_j^t R\theta_j \right)^{1/2} = \max_{j \in [l-1]} \left( \frac{1}{|B|} \sum_{b \in B}\ip{\phi(b),\theta_j}^2  \right)^{1/2}.
$$
\end{definition}

We define two related notions. Balls around zero in the $R$-norm are defined as
$$
\B_R(\eps) := \{\Theta \in \R^{(l-1)A} : \|\Theta\|_R \leq  \eps \}.
$$
The Voronoi cell of $0$ in the $R$-norm, with respect to the dual lattice $L(\Phi)$, is
$$
D := \Big\{\Theta \in \R^{(l-1)A}: \|\Theta\|_{R} < \|\Theta-\alpha\|_R \quad \forall \alpha \in L(\Phi) \setminus \{0\} \Big\}.
$$
Observe that $D$ is a fundamental domain of $L(\Phi)$ up to a set of measure zero (its boundary), which we can ignore in our calculations.
Then we have the following Fourier inversion formula over lattices: for every $\Gamma \in \L(\Phi)$ it holds that
\begin{equation}
\label{eq:fourier}
\Pr[X=\Gamma] = \frac{1}{\vol(D)}\int_{D} \widehat{X}(\Theta)e^{-2\pi i\ip{ \Gamma,\Theta }}d\Theta  = \det(\L(\Phi))\int_{D} \widehat{X}(\Theta)e^{-2\pi i\ip{\Gamma,\Theta}}d\Theta.
\end{equation}
Note that this formula holds true for any fundamental region of $L(\Phi)$ but we chose it to be the Voronoi cell $D$ arising from the norm $\|\cdot\|_{R}$ because it would help in the computations later on.
In order to prove \eqref{eq:prob_positive}, we specialize \eqref{eq:fourier} to $\Gamma=\E[X]$ and obtain
\begin{equation}
\label{eq:fourier_mean}
\Pr[X=\E[X]] = \det(\L(\Phi))\int_{D} \widehat{X}(\Theta)e^{-2\pi i\ip{\E[X],\Theta}}d\Theta.
\end{equation}
In the next section, we approximate this by a Gaussian estimate.

\section{Gaussian estimate}
\label{sec:Gaussian}

In order to estimate \eqref{eq:fourier_mean}, let $Y$ be a Gaussian random variable in $\R^{(l-1)|A|}$ with the same mean and covariance as $X$.
The density $f_Y$ of $Y$ is given by
\begin{equation}\label{eq:Y_density}
f_Y(x) = \frac{\exp(-\frac{1}{2}(x-\E[X])^t\Sigma[X]^{-1}(x-\E[X]))}{(2\pi)^{\frac{(l-1)|A|}{2}}\sqrt{\det (\Sigma[X])}}.
\end{equation}
The Fourier transform of $Y$ equals
\begin{equation}\label{eq:Y_fourier}
\widehat{Y}(\Theta) := \E[e^{2\pi i\ip{ Y,\Theta}}] = e^{2\pi i\ip{\E[X],\Theta}-2\pi^2\Theta^t\Sigma[X]\Theta}.
\end{equation}
The inverse Fourier transform applied to $Y$ yields
\begin{equation}\label{eq:Y_inverse}
f_Y(x) = \int_{\R^{(l-1)A}} \widehat{Y}(\Theta)e^{-2\pi i\ip{ x,\Theta }}d\Theta \qquad \forall x \in \R^{(l-1)A}.
\end{equation}
We show that $\Pr[X=\E[X]]$ can be approximated by an appropriate scaling of $f_Y(\E[X])$. By \eqref{eq:fourier_mean} we have
$$
\frac{\Pr[X=\E[X]]}{\det(\L(\Phi))} - f_{Y}(\E[X]) = \int_{D} \widehat{X}(\Theta)e^{-2\pi i\ip{\E[X],\Theta}}d\Theta -
\int_{\R^{(l-1)A}} \widehat{Y}(\Theta)e^{-2\pi i\ip{ \E[X],\Theta }}d\Theta.
$$
Note that by plugging $x=\E[X]$ in \eqref{eq:Y_density} we obtain that
\begin{equation}\label{eq:Y_density_mean}
f_Y(\E[X]) = \frac{1}{(2\pi)^{\frac{(l-1)|A|}{2}}\sqrt{\det (\Sigma[X])}}.
\end{equation}
We will show that $|\frac{\Pr[X=\E[X]]}{\det(\L(\Phi))} - f_{Y}(\E[X])| \ll f_Y(\E[X])$.  For $\eps>0$ to be chosen later, we will bound it by
\begin{align}
\label{eq:def_Is}
&\left|\frac{\Pr[X=\E[X]]}{\det(\L(\Phi))} - f_{Y}(\E[X])\right| \leq \nonumber\\
&\underbrace{\int_{\B_R(\eps)}|\widehat{X}(\Theta) - \widehat{Y}(\Theta)|d\Theta}_{= I_1} + \underbrace{\int_{D \setminus \B_R(\eps)}|\widehat{X}(\Theta)|d\Theta}_{= I_2} + \underbrace{\int_{\R^{(l-1)A} \setminus \B_R(\eps)}|\widehat{Y}(\Theta)|d\Theta}_{= I_3}.
\end{align}
At a high level, the upper bound is obtained by comparing $\widehat{X}(\Theta)$ and $\widehat{Y}(\Theta)$ in a small
enough ball; and upper bounding their absolute value outside this ball. Observe that we need $\eps$ to be small enough so that $\B_R(\eps) \subset D$.

\subsection{Norms on $\R^{|A|}$ induced by $\phi$}
The following key technical lemmas from \cite{KLP12} are very useful in bounding the integrals. We begin with defining a few norms which are all functions of $\phi$.

\begin{definition}[Norms on $\R^{|A|}$ induced by $\phi$]
For $\theta \in \R^{|A|}$ define the following norms:
\begin{itemize}
\item $\ii{\theta}_{\phi,\infty} = \max_{b\in B}|\ip{\phi(b),\theta}|$.
\item $\ii{\theta}_{\phi,2} = \left( \frac{1}{|B|} \sum_{b\in B}|\ip{\phi(b),\theta}|^2 \right)^{1/2}$.
\end{itemize}
Furthermore, for $b \in B$ let $\ip{\phi(b),\theta}=n_b + r_b$ where $n_b \in \Z$ and $r_b \in [-1/2,1/2)$. Define
\begin{itemize}
\item $\iii{\theta}_{\phi,\infty} = \max_{b\in B}|r_b|$.
\item $\iii{\theta}_{\phi,2} = \left(\frac{1}{|B|}  \sum_{b\in B}|r_b|^2 \right)^{1/2}$.
\end{itemize}
\end{definition}

Note that if $\theta' \in L(\phi)$ then $\ip{\phi(b),\theta+\theta'} - \ip{\phi(b),\theta} \in \Z$ for all $b \in B$. In particular,
$\iii{\theta+\theta'}_{\phi,\infty}=\iii{\theta}_{\phi,\infty}$ and $\iii{\theta+\theta'}_{\phi,2}=\iii{\theta}_{\phi,2}$.
The following lemmas from \cite{KLP12} relate the above norms.

\begin{lemma}[Lemma 4.4 in \cite{KLP12}]
\label{lem:M}
For every $\theta \in \R^{A}$ it holds that
$$
\ii{\theta}_{\phi,\infty} \le M \ii{\theta}_{\phi,2}
$$
and
$$
\iii{\theta}_{\phi,\infty} \le M \iii{\theta}_{\phi,2}.
$$
Here, $M := C\left(|A|\log(2c_2|A|) \right)^{3/2}$ for some absolute constant $C>0$.
\end{lemma}

\begin{lemma}[Claim 4.12 in \cite{KLP12}]
\label{lem:M2}
Assume that for $\theta \in \R^{A}$ it holds that
$$
\iii{\theta}_{\phi,\infty} < \frac{1}{c_3}.
$$
Then there exists $\theta' \in L(\phi)$ such that $\ip{\theta-\theta', \phi(b)} \in [-1/2,1/2)$ for all $b \in B$.
In particular
$$
\ii{\theta-\theta'}_{\phi,2} = \iii{\theta}_{\phi,2}.
$$
\end{lemma}

\subsection{Norms on $\R^{(l-1)|A|}$ induced by $\Phi$}

We extend the previous definitions to norms on $\R^{(l-1)|A|}$ induced by $\Phi$, and prove related lemmas relating the different norms.

\begin{definition}[Generalizing the norms to $\R^{(l-1)|A|}$]
For $\Theta=(\theta_1,\ldots,\theta_{l-1}) \in \R^{(l-1)|A|}$ define the following norms:
\begin{itemize}
\item $\ii{\Theta}_{\Phi,\infty} = \max_{j \in [l-1]} \ii{\theta_j}_{\phi,\infty}$
\item $\ii{\Theta}_{\Phi,2} = \max_{j \in [l-1]}\ii{\theta_j}_{\phi,2}$
\item $\iii{\Theta}_{\Phi,\infty} = \max_{j \in [l-1]} \iii{\theta_j}_{\phi,\infty}$
\item $\iii{\Theta}_{\Phi,2} = \max_{j \in [l-1]}\iii{\theta_j}_{\phi,2}$
\end{itemize}
\end{definition}

Observe that $\|\cdot\|_{\Phi,2}$ is the same as the $R$-norm $\|\cdot\|_R$ we defined before.
Similar to before, if $\Theta' \in L(\Phi)$ then
$\iii{\Theta+\Theta'}_{\Phi,\infty}=\iii{\Theta}_{\Phi,\infty}$ and $\iii{\Theta+\Theta'}_{\Phi,2}=\iii{\Theta}_{\Phi,2}$.

The following extends Lemma~\ref{lem:M} and Lemma~\ref{lem:M2} to the norms induced by $\Phi$.
\begin{lemma}\label{lem:MM}
For the same $M$ defined in Lemma~\ref{lem:M}, for every $\Theta \in \R^{(l-1)|A|}$ it holds that
$$
\ii{\Theta}_{\Phi,\infty} \le M \ii{\Theta}_{\Phi,2}
$$
and
$$
\iii{\Theta}_{\Phi,\infty} \le M \iii{\Theta}_{\Phi,2}.
$$
\end{lemma}

\begin{proof}
Let $\Theta = (\theta_1,\ldots,\theta_{l-1})$. Then using Lemma~\ref{lem:M} we have
$$
\ii{\Theta}_{\Phi,\infty} = \max_{j \in [l-1]} \ii{\theta_j}_{\phi,\infty} \leq \max_{j \in [l-1]} M\ii{\theta_j}_{\phi,2} = M\ii{\Theta}_{\Phi,2}
$$
and
$$
\iii{\Theta}_{\Phi,\infty} = \max_{j \in [l-1]} \iii{\theta_j}_{\phi,\infty} \leq \max_{j \in [l-1]} M\iii{\theta_j}_{\phi,2} = M\iii{\Theta}_{\Phi,2}.
$$
\end{proof}

\begin{lemma}
\label{lem:MM2}
Assume that for $\Theta \in \R^{(l-1)A}$ it holds that
$$
\iii{\Theta}_{\Phi,\infty} < \frac{1}{c_3}.
$$
Then there exists $\Theta' \in L(\Phi)$ such that $\ip{\Theta-\Theta', \Phi(b,j)} \in [-1/2,1/2)$ for all $b \in B, j \in [l-1]$.
In particular
$$
\ii{\Theta-\Theta'}_{\Phi,2} = \iii{\Theta}_{\Phi,2}.
$$
\end{lemma}

\begin{proof}
Let $\Theta = (\theta_1,\ldots,\theta_{l-1})$. We have $\iii{\theta_j}_{\phi,\infty} < \frac{1}{c_3}$ for all $j \in [l-1]$.
Then using Lemma~\ref{lem:M2} we get that there exist $\theta'_1,\ldots,\theta'_{l-1} \in L(\phi)$ such that
$\ip{\theta_j-\theta'_j, \phi(b)} \in [-1/2,1/2)$ for all $b \in B$. The lemma follows for $\Theta' = (\theta'_1,\ldots,\theta'_{l-1}) \in L(\Phi)$.
\end{proof}

\subsection{Estimates for balls in the Voronoi cell}

To recall, we need $\eps>0$ to be small enough so that $\B_R(\eps)$ is contained in the Voronoi cell $D$. The following Lemma utilizes Lemma~\ref{lem:MM} to achieve that.

\begin{lemma}
\label{lem:ball_in_D}
If $\eps < \frac{1}{2M}$ then $\B_R(\eps) \subset D$.
\end{lemma}

\begin{proof}
Let $\Theta=(\theta_1,\ldots,\theta_{l-1})\in L(\Phi) \setminus \{0\}$. By definition $\ip{ \phi(b),\theta_j } \in \mathbb{Z}$ for all $b \in B, j \in [l-1]$.
Since $\L(\phi)$ is of full rank and $\Theta \neq 0$, there exists some $b \in B, j \in [l-1]$ for which $|\ip{ \phi(b),\theta_j }| \geq 1$. Thus
$$
\ii{\Theta}_{\Phi,\infty} \ge 1.
$$
By Lemma~\ref{lem:MM} if follows that
$$
\|\Theta\|_R = \|\Theta\|_{\Phi,2} \ge 1/M.
$$
Thus, if $\Theta' \in \B_R(\eps)$ for $\eps<1/2M$ then
$$
\|\Theta-\Theta'\|_R \ge \|\Theta\|_R-\|\Theta'\|_R \ge 1/M - \eps > 1/2M \ge \|\Theta'\|_R.
$$
Hence $\B_R(\eps) \subset D$ for any $\eps < \frac{1}{2M}$.
\end{proof}

Let $\Theta \in D \setminus \B_R(\eps)$. Clearly, its $\ii{\cdot}_{\Phi,2}$ norm is noticeable (at least $\eps$). We show that also its
$\iii{\cdot}_{\Phi,2}$ norm is noticeable. This will later be useful in bounding $\hat{X}(\Theta)$ in $D \setminus \B_R(\eps)$.

\begin{lemma}
\label{lem:norms_in_D}
Assume that $c_3 \ge 2$ and $\eps < 1/ c_3 M$.
Let $\Theta \in D \setminus \B_R(\eps)$. Then $\iii{\Theta}_{\Phi,2} > \eps$.
\end{lemma}

\begin{proof}
Note that the condition of Lemma~\ref{lem:ball_in_D} hold, and so $\B_R(\eps) \subset D$.
Assume towards contradiction that $\iii{\Theta}_{\Phi,2} \le \eps$.
Applying Lemma~\ref{lem:MM} gives $\iii{\Theta}_{\Phi,\infty} \le \eps M < \frac{1}{c_3}$.
Applying Lemma~\ref{lem:MM2}, this implies that there exists $\Theta' \in L(\Phi)$ for which
$\ii{\Theta-\Theta'}_{\Phi,2}=\iii{\Theta}_{\Phi,2} \le \eps$. However, as $\Theta \in D$ we
have $\ii{\Theta}_{\Phi,2} \le \ii{\Theta-\Theta'}_{\Phi,2} \le \eps$, which gives that $\Theta \in \B_R(\eps)$, a contradiction.
\end{proof}

\subsection{Bounding the integrals}
The following lemmas provide the necessary bounds on the integrals $I_1,I_2,I_3$, as defined in \eqref{eq:def_Is}. The proofs are deferred to Section~\ref{sec:integrals}.

\begin{lemma}
\label{lemma:I1}
Assume that $\eps \le (C M |B|l)^{-1/3}$. Then
$$
I_1 \leq \frac{C l^3 M |A|^{3/2}}{|B|^{1/2}} \cdot f_Y(\E[X]).
$$
Here $C>0$ is some large enough absolute constant.
\end{lemma}

\begin{lemma}
\label{lemma:I2}
Assume that $c_3 \ge 2$ and $\eps < 1/ c_3 M$. Then
$$
I_2 \le \frac{1}{\det(\L(\Phi))} \exp \left(- \frac{|B| \eps^2}{l^2} \right).
$$
\end{lemma}

\begin{lemma}
\label{lemma:I3}
For any $\eps>0$ it holds that
$$
I_3 \leq f_Y(\E[X]) \cdot (l-1) 2^{|A|/2} \exp\left(-\frac{\pi^2|B|\eps^2}{l^2} \right).
$$
\end{lemma}

\subsection{Putting it all together}
Let $C_1,C_2,\ldots$ be unspecified absolute constants below. By choosing an appropriate basis for $V$ which is $c_2$-bounded in $\ell_{\infty}$, we may assume that $\phi:B \to \Z^A$
where $|\phi(b)_a| \le c_2$ for all $a \in A, b \in B$.

Set $\eps =  (C_1 M |B|)^{-1/3}$
so that we may apply Lemma~\ref{lemma:I1}, and assume that $\eps \le 1/c_3 M$ so that we may apply Lemma~\ref{lemma:I2}.
We thus have
$$
\Pr[X = \E[X]] = \det(\L(\Phi)) f_Y(\E[X]) (1 + \alpha_1 + \alpha_3) + \alpha_2,
$$
where
\begin{align*}
&|\alpha_1| \le \frac{C_1 l^3 M |A|^{3/2}}{|B|^{1/2}}, \\
&|\alpha_2| \le \exp\left(- \frac{|B| \eps^2}{l^2}\right) = \exp\left(- C_2 \frac{|B|^{1/3}}{ l^2 M^{2/3}}\right),\\
&|\alpha_3| \le (l-1) 2^{|A|/2} \exp\left(-\frac{\pi^2|B|\eps^2}{l^2}\right) \le
l 2^{|A|} \exp\left(-C_3 \frac{|B|^{1/3}}{l^2 M^{2/3}}\right).
\end{align*}
We would like that $|\alpha_1|, |\alpha_3| \le 1/4$, which requires that
$$
|B| \ge C_4 |A|^3 M^2 l^6 c_3^3
$$
Thus
$$
\Pr[X = \E[X]] \ge \frac{1}{2} \det(\L(\Phi)) f_Y(\E[X]) + \alpha_2.
$$
We assume that $\phi:B \to \Z^A$, so $\L(\Phi)$ is an integer lattice and hence $\det(\L(\Phi)) \ge 1$. We next lower bound $f_Y(\E[X])$.
We have by \eqref{eq:Y_density_mean} that
$$
f_Y(\E[X]) =
\frac{1}{(2\pi)^{\frac{(l-1)|A|}{2}}\sqrt{\det (\Sigma[X])}}.
$$
We assume that $\phi$ is spanned by integer vectors of maximum entry at most $c_2$, so we can bound each entry of $\Sigma[X]$ by
$$
|\Sigma[X]_{(a,i),(a',i')}| \le \sum_{b \in B} |\phi(b)_a \phi(b)_{a'}|  \le |B| c_2^2.
$$

Thus using the Hadamard bound we have
$$
\det (\Sigma[X]) \le \left(\sqrt{(l-1)|A|}|B| c_2^2\right)^{(l-1)|A|}.
$$
In order to require $|\alpha_2| \le (1/4) f_Y(\E[X])$, say, we need to require that
$$
|B| \ge C_5 |A|^3 M^2 l^7 \log(|A| M l).
$$
Putting it all together, and plugging in the value of $M$ from Lemma~\ref{lem:M},
as long as
$$
|B| \ge C |A|^6 l^7 c_3^3 \log^3(|A| c_2 c_3 l),
$$
we have that
$$
\Pr[X = \E[X]] \ge \frac{1}{4} \det(\L(\Phi)) f_Y(\E[X]) > 0.
$$

\section{Bounding the integrals}
\label{sec:integrals}

\subsection{Bounding $I_1$}
Recall that $I_1 = \int_{\B_R(\eps)} |\hat{X}(\Theta) - \hat{Y}(\Theta)| d \Theta$. We will bound it by bounding pointwise the difference $|\hat{X}(\Theta) - \hat{Y}(\Theta)|$
and integrating it.

We first compute an exact formula for $\widehat{X}(\Theta)$. Recall that $X = \sum_{b \in B} \Phi(b,\tau(b))$ where $\tau(b) \in [l]$ are independently and uniformly chosen. Thus
\begin{equation}\label{eq:fourier_expression}
\widehat{X}(\Theta) = \E_{X} \left[ e^{2 \pi i \ip{X, \Theta}} \right] =
\prod_{b\in B}\left[ \frac{1}{l}\left(1 + \sum_{j=1}^{l-1}e^{2\pi i\ip{ \phi(b),\theta_j}} \right)\right].
\end{equation}

Fix $\Theta=(\theta_1,\ldots,\theta_{l-1})$. To simplify notations, let $x_{b,j} = 2\pi\ip{\phi(b),\theta_j}$ and
$\x_b = (x_{b,1},\dots .x_{b,l-1}) \in \R^{l-1}$.
Define the function $f:\R^{l-1}\rightarrow\mathbb{C}$ given by $f(\x) = \frac{1}{l}\left(1 + \sum_{j=1}^{l-1}e^{ix_j} \right)$. Then
we can simplify \eqref{eq:fourier_expression} as
\begin{equation}\label{eq:fourier_expression2}
\widehat{X}(\Theta) = \prod_{b\in B}f(\x_b).
\end{equation}

We next approximate $\log f(\x)$. We use the shorthand $O(z)$ to denote a (possibly complex) value, whose absolute value is bounded by $C z$ for some
unspecified absolute constant $C>0$. For $\x=(x_1,\ldots,x_{l-1})$ we denote $|\x| = \max_j |x_j|$.

\begin{claim}
\label{claim:approx_f}
Let $\x=(x_1,\ldots,x_{l-1}) \in \R^{l-1}$ with $|\x| \le 1$. Then
$$
f(\x) = \exp \left( i \frac{1}{l}\sum_{j}x_j - \frac{1}{2l}\left(1-\frac{1}{l}\right)\sum_{j}x_j^2 + \frac{1}{2l^2}\sum_{j\neq j'}x_jx_{j'} + O\left(|\x|^3 \right) \right).
$$
\end{claim}

\begin{proof}
Let $y = \frac{1}{l} \sum_{j=1}^{l-1}(e^{ix_j}-1) $ so that $f(\x) = 1+y$. The condition $|\x| \le 1$ guarantees that $|y|<1$, so the Taylor expansion for $\log(1+y)$ converges and gives
$$
\log(f(\x)) = \log(1+y) = y -\frac{y^2}{2}+O(|y|^3).
$$
One can verify that $|y| \le O(|\x|)$, that
$$
y = i \frac{1}{l}\sum_{j}x_j -\frac{1}{2l}\sum_{j}x_j^2 + O\left (|\x|^3 \right).
$$
and that
$$
y^2 = - \frac{1}{l^2}\left(\sum_{j}x_j\right)^2+O\left (|\x|^3 \right).
$$
Combining these gives the required result.
\end{proof}

Applying Claim~\ref{claim:approx_f} to \eqref{eq:fourier_expression2} allows us to approximate $\widehat{X}(\Theta)$ as
$$
\widehat{X}(\Theta) = \exp\left( \frac{2\pi i}{l}\sum_{\substack{b\in B\\j \in [l-1]}}\ip{\phi(b),\theta_j} -\frac{2\pi^2}{l}(1-\frac{1}{l})\sum_{\substack{b\in B\\j \in [l-1]}}\ip{\phi(b),\theta_j}^2 + \frac{2\pi^2}{l^2}\sum_{\substack{b\in B\\ j\neq j'}} \ip{\phi(b),\theta_j}\ip{\phi(b),\theta_{j'}} + \delta(\Theta) \right),
$$
which can be rephrased as
\begin{equation}
\label{eq:fourier_expression3}
\widehat{X}(\Theta) = \exp\left(2\pi i\ip{\E[X],\Theta} - 2\pi^2\Theta^t\Sigma[X]\Theta + \delta(\Theta)\right).
\end{equation}

The error term $\delta(\Theta)$ is bounded by
\begin{align*}
\delta(\Theta)
&= O \left( \sum_{b \in B} |\x_b|^3 \right)
= O \left( \sum_{b \in B} \max_{j \in [l-1]} |\ip{\phi(b), \theta_j}|^3 \right) \\
& \le O \left( \max_{b \in B, j \in [l-1]} |\ip{\phi(b), \theta_j}| \right) \left( \sum_{b \in B} \max_{j \in [l-1]} |\ip{\phi(b), \theta_j}|^2 \right)\\
&= O \left( \|\Theta\|_{\Phi,\infty} \cdot |B|l \|\Theta\|_{\Phi,2}^2 \right)
\end{align*}
By Lemma~\ref{lem:MM} we have $\|\Theta\|_{\Phi,\infty} \le M \|\Theta\|_{\Phi,2}$, and hence as $\|\Theta\|_{\Phi,2}=\|\Theta\|_R$
we conclude that
\begin{equation}\label{eq:bound_delta}
|\delta(\Theta)| \le C_1 M |B|l \|\Theta\|_R^3,
\end{equation}
where $C_1>0$ is some absolute constant.

Next, we apply these estimates to bound the integral $I_1$.
Recall that by \eqref{eq:Y_fourier} we have
$$
\widehat{Y}(\Theta) := \exp(2\pi i\ip{\E[X],\Theta}-2\pi^2\Theta^t\Sigma[X]\Theta).
$$
Thus we can bound $I_1$ by
$$
I_1
= \int_{\B_R(\eps)} |\hat{X}(\Theta) - \hat{Y}(\Theta)| d \Theta
 \le \int_{\B_R(\eps)}e^{-2\pi^2\Theta^t\Sigma[X]\Theta}|e^{\delta(\Theta)}-1|d\Theta.
$$
We assume that $\eps>0$ is small enough so that $C_1 M |B|l\eps^3 \leq 1$, so that
for all for $\Theta \in \B_R(\eps)$ we have
$$
|e^{\delta(\Theta)}-1| \leq 2 \delta(\Theta) \le 2 C_1 M |B|l \|\Theta\|_R^3.
$$
Thus
$$
I_1 \leq 2C_1 M|B|l\int_{\B_R(\eps)}e^{-2\pi^2\Theta^t\Sigma[X]\Theta}\|\Theta\|_R^{3} d\Theta \leq 2C_1 M|B|l\int_{\R^{(l-1)A}}e^{-2\pi^2\Theta^t\Sigma[X]\Theta}\|\Theta\|_R^{3} d\Theta.
$$

Next, we evaluate the integral on the right. Let $Z$ be a Gaussian random variable in $\R^{(l-1)|A|}$ with mean zero and covariance matrix $\frac{1}{4 \pi^2} \Sigma[X]^{-1}$. Then
the density of $Z$ is
$$
f_{Z}(\Theta) = (2\pi)^{\frac{(l-1)|A|}{2}}\sqrt{\det(\Sigma)}e^{-2\pi^2\Theta^t\Sigma[X]\Theta} = \frac{1}{f_Y(\E[X])} e^{-2\pi^2\Theta^t\Sigma[X]\Theta},
$$
where we have used \eqref{eq:Y_density_mean}.
Hence
$$
\int_{\R^{(l-1)A}}e^{-2\pi^2\Theta^t\Sigma[X]\Theta}\|\Theta\|_R^{3} d\Theta = f_Y(\E[X]) \cdot \E [\|Z\|_R^3].
$$

Let $G \in \R^{(l-1) |A|}$ be a standard multivariate Gaussian with mean zero and identity covariance matrix.
Recall that by Claim~\ref{claim:covariance} we have $\Sigma[X] = R \otimes M$, where $M$ has eigenvalues $(1/l^2,1/l,\ldots,1/l)$. In particular, $\Sigma[X]$ is positive definite, so its root exists and is equal to $\Sigma[X]^{1/2} = R^{1/2} \otimes M^{1/2}$. Hence we have $Z = \frac{1}{2\pi}(R^{-1/2} \otimes M^{-1/2})G$. Denoting $G=(G_1,\ldots,G_{l-1})$ with $G_i \in \R^{|A|}$ and similarly $Z=(Z_1,\ldots,Z_{l-1})$ with $Z_i \in \R^{|A|}$, we have
\begin{align*}
\E_Z\left[\|Z\|_R^3\right]
&= \E_Z\left[\max_{j}\left(\frac{1}{|B|}Z_j^tRZ_j\right)^{3/2}\right] \\
&\leq \E_Z\left[\left(\sum_{j}\frac{1}{|B|}Z_j^tRZ_j\right)^{3/2}\right] \\
&= \E_Z\left[\left(\frac{1}{|B|} Z^t (R\otimes I) Z\right)^{3/2}\right]\\
&= \E_Z\left[\left(\frac{1}{4\pi^2|B|} G^t (I \otimes M^{-1}) G\right)^{3/2}\right]\\
&= \left(\left(\frac{l^2}{4\pi^2|B|}\right)^{\frac{3}{2}} + (l-2)\left(\frac{l}{{4\pi^2|B|}}\right)^{\frac{3}{2}}\right)\E\left[\|G\|_2^3\right]\\
&\leq \frac{2l^3}{(4\pi^2)^{3/2}|B|^{3/2}}\E\left[\|G\|_2^3\right].
\end{align*}
Note that by Jensen's inequality $\E[\|G\|_2^3] \leq \E[\|G\|_2^4]^{3/4} \leq 4^{3/4}|A|^{3/2}$. Thus we can summarize that
$$
I_1 \leq O \left( \frac{ l^4 M |A|^{3/2}}{|B|^{1/2}} \right) \cdot f_Y(\E[X]).
$$

\subsection{Bounding $I_2$}
Recall that $I_2 = \int_{D \setminus \B_R(\eps)}|\widehat{X}(\Theta)|d\Theta$. We upper bound $I_2$ by proving an upper bound on $|\hat{X}(\Theta)|$ in $D \setminus \B_R(\eps)$.

Fix $\Theta=(\theta_1,\ldots,\theta_{l-1}) \in D$ where we assume $\ii{\Theta}_{\Phi,2} = \ii{\Theta}_R \ge \eps$.
Our goal is to upper bound $\hat{X}(\Theta)$.
Let $\ip{\phi(b),\theta_j} = n_{b,j} + r_{b,j}$ where $n_{b,j} \in \mathbb{Z}$ and $r_b \in [-1/2,1/2)$. By \eqref{eq:fourier_expression2} we have
$$
\widehat{X}(\Theta) = \prod_{b\in B}\left[ \frac{1}{l}\left(1 + \sum_{j=1}^{l-1}e^{2\pi i\ip{ \theta_j,\phi(b)}} \right)\right] =
\prod_{b\in B}\left[ \frac{1}{l}\left(1 + \sum_{j=1}^{l-1}e^{2\pi i \cdot r_{b,j}} \right)\right] =
\prod_{b\in B} f(2 \pi \cdot \r_b),
$$
where $f(\x) = \frac{1}{l}\left(1 + \sum_{j=1}^{l-1}e^{ix_j} \right)$ and $\r_b=(r_{b,1},\ldots,r_{b,l-1})$. Recall that $|\x| = \max |x_j|$.

\begin{claim}
\label{claim:f_bound}
Let $\x \in \R^{l-1}$ be with $|\x| \le \pi$. Then $|f(\x)| \le \exp(-|\x|^2 / 8l)$.
\end{claim}

\begin{proof}
Let $x_j = |\x|$. Then $|f(\x)| \le \frac{l-2}{l} + \frac{2}{l} |\frac{1+e^{i x_j}}{2}|$.
If $z \in [-\pi,\pi]$ then  $|\frac{1 + e^{i z}}{2}| \le e^{-z^2 / 8}$. One can verify that
$$
\log |f(\x)| \le \log\left(1 - \frac{2}{l} \left(e^{-|\x|^2/8} -1 \right) \right) \le - \frac{|\x|^2}{ 8 l}.
$$
\end{proof}
Thus we have
$$
\log |\widehat{X}(\Theta)| \le -\frac{4 \pi^2}{8l} \sum_{b \in B} |\r_b|^2 \le  -\frac{1}{l^2} \sum_{b \in B, j \in [l-1]} r_{b,j}^2 =
-\frac{|B|}{l^2} \iii{\Theta}_{\Phi,2}^2.
$$
Next, assume that $\eps < 1/ c_3 M$. By Lemma~\ref{lem:norms_in_D} we have that $\iii{\Theta}_{\Phi,2} \ge \eps$. Thus
$$
|\widehat{X}(\Theta)| \le \exp(- |B| \eps^2/l^2).
$$
Thus we may bound
$$
I_2 \le \vol(D) \exp(- |B| \eps^2/l^2) = \frac{1}{\det(\L(\Phi))} \exp(- |B| \eps^2/l^2).
$$

\subsection{Bounding $I_3$}
Recall that
$$
I_3 = \int_{\R^{(l-1)A}\setminus\B_R(\eps)}|\hat{Y}(\Theta)| d\Theta = \int_{\R^{(l-1)A}\setminus\B_R(\eps)}e^{-2\pi^2\Theta^t\Sigma[X]\Theta}d\Theta.
$$
Similar to the calculation of the bound for $I_1$, let $Z \in \R^{(l-1)|A|}$ be a Gaussian random variable with mean zero and covariance matrix $\frac{1}{4 \pi^2} \Sigma[X]^{-1}$.
Then
$$
I_3 = f_Y(\E[X]) \cdot \Pr\left[\|Z\|_R > \eps \right].
$$

Recall that we showed that if we set $Z=(Z_1,\ldots,Z_{l-1})$, then $Z_1,\ldots,Z_{l-1} \in \R^A$ are independent Gaussian random variables with mean zero,
where $Z_1$ has covariance matrix $\frac{l^2}{4\pi^2}R^{-1}$ and $Z_j$ has covariance matrix $\frac{l}{4\pi^2}R^{-1}$ for $j=2,\ldots,l-1$. We may thus bound

\begin{align*}
\Pr\left[\|Z\|_R > \eps \right] &= \Pr_Z\left[\max_{j}\left(\frac{1}{|B|}Z_j^tRZ_j \right) > \eps^2 \right] \le
\sum_{j}\Pr_{Z_j}\left[\left(\frac{1}{|B|}Z_j^tRZ_j \right) > \eps^2 \right]\\
&=\Pr_{G'}\left[\|G'\|_2^2 > \frac{4\pi^2|B|\eps^2}{l^2} \right] + (l-2)\Pr_{G'}\left[\|G'\|_2^2 > \frac{4\pi^2|B|\eps^2}{l} \right]\\
&\le (l-1) \Pr_{G'}\left[\|G'\|_2^2 > \frac{4\pi^2|B|\eps^2}{l^2} \right],
\end{align*}
where $G' \in \R^A$ is a Gaussian random variable with mean zero and identity covariance matrix.

In order to bound $\Pr_{G'}\left[\|G'\|_2^2 > \rho\right]$ we note that for any $t<1/2$, it holds that
$\E \left[ e^{t\|G'\|_2^2} \right] = (1-2t)^{-|A|/2}$. Fixing $t=1/4$ and applying Markov's inequality gives
$$
\Pr_{G'}\left[\|G'\|_2^2 > \rho\right] \leq \frac{\E \left[ e^{\|G'\|_2^2/4} \right]}{e^{\rho/4}} = 2^{|A|/2}e^{-\rho/4}.
$$
So
$$
I_3 \leq f_Y(\E[X]) \cdot (l-1) 2^{|A|/2} e^{-\frac{\pi^2|B|\eps^2}{l^2}}.
$$

\section*{Acknowledgment}\vspace{-.75ex}
We are grateful to Yeow Meng Chee and Tuvi Etzion
for helpful discussions regarding the history
and the current state of knowledge about
the existence problem for large sets of designs. The research of Shachar Lovett
was supported by the National Science Foundation under grant CCF-1614023.
The research of Sankeerth Rao and Alexander Vardy was
supported by the National Science Foundation under grants CCF-1405119 and CCF–1719139.

\bibliographystyle{abbrv}

\end{document}